\documentclass{amsart}
\usepackage{changebar}
\usepackage{lineno}
\usepackage{amsthm}
\usepackage{verbatim}
\usepackage[utf8]{inputenc}
\usepackage{blindtext}
\usepackage{bm}

\makeatletter
\renewcommand\part{%
   \if@noskipsec \leavevmode \fi
   \par
   \addvspace{4ex}%
   \@afterindentfalse
   \secdef\@part\@spart}

\def\@part[#1]#2{%
    \ifnum \c@secnumdepth >\m@ne
      \refstepcounter{part}%
      \addcontentsline{toc}{part}{\thepart\hspace{1em}#1}%
    \else
      \addcontentsline{toc}{part}{#1}%
    \fi
    {\parindent \z@ \raggedright
     \interlinepenalty \@M
     \normalfont
     \ifnum \c@secnumdepth >\m@ne
       \Large\bfseries \partname\nobreakspace\thepart
       \par\nobreak
     \fi
     \huge \bfseries #2%
     \par}%
    \nobreak
    \vskip 3ex
    \@afterheading}
\def\@spart#1{%
    {\parindent \z@ \raggedright
     \interlinepenalty \@M
     \normalfont
     \huge \bfseries #1\par}%
     \nobreak
     \vskip 3ex
     \@afterheading}
\makeatother

\newtheorem{lemma}{Lemma}
\newtheorem{theorem}{Theorem}

\newcommand{\nmid}{\not \hspace{0.25em} \mid}

\begin{document}

\title[The diophantine equation $(2^{k}-1)(b^{k}-1)=y^{q} $  ] {The diophantine equation $(2^{k}-1)(b^{k}-1)=y^{q} $ } 

\author{Chang Liu} 
\address[Chang Liu]{Mathematisches Institut der Universit\"at 
G\"ottingen, Bunsenstr. 3-5, DE-37073, G\"ottingen, Germany}
\email[Chang Liu]{chang.liu@mathematik.uni-goettingen.de} 
\author{Bo He} 
\address[Bo He]{1. Mathematisches Institut der Universit\"at 
G\"ottingen, Bunsenstr. 3-5, DE-37073, G\"ottingen, Germany; 2. Applied Mathematics Institute of Aba Teachers University, Wenchuan, Sichuan, 623002, P. R. China}
\email[Bo He]{bo.he@mathematik.uni-goettingen.de,  bhe@live.cn}

\vspace{0.3cm}

\begin{abstract}
In this paper, we consider the exponential Diophantine equation
\(
(2^k-1)(b^k-1)=y^q
\)
with $k\ge 2$, an odd integer $b$ and an odd prime exponent $q$. We obtain effective 
upper bounds for $q$ in terms of $b$. In particular, we show that $q\le \log_2(b+1)$ holds
apart from a finite, explicitly determined set of exceptional pairs $(b,q)$  with $3\le b<10^6$.
As an application, we prove that the related equation
\(
(2^k-1)(b^k-1)=x^n
\)
has no positive integer solution $(k,x,n)$ for several specific values of
$b$, including $b\in\{5,7,11,13,21,23,27,29\}$ except for $(2^2-1)(7^2-1)=12^2$. 
\end{abstract}

\maketitle
\vspace*{-0.5cm}



\section{Introduction}

 In 2002, Szalay \cite{Szalay} investigated the Diophantine equation
$(2^n - 1)(3^n - 1) = x^2$
and proved that there is no solution in positive integers \( n \) and \( x \). This was 
the first investigation of a particular case of the Diophantine equation
\begin{equation}\label{eq:ab2}
(a^n - 1)(b^n - 1) = x^2.
\end{equation}
Moreover, he studied similar equations such as \( (2^n - 1)(5^n - 1) = x^2 \), demonstrating that they have only bounded solutions with specific values of \( n \). In
the same year, Hajdu and Szalay \cite{Hajdu-Szalay} showed that there is no solution for \((2^n - 1)(6^n - 1) = x^2\). 

 Cohn \cite{Cohn} studied the equation \((a^n - 1)(b^n - 1) = x^2\)
and explored the integer solutions for given values of \(a\) and \(b\), deriving general results and conjectures about this equation in 2001. Subsequently, in 2002, Luca and Walsh \cite{Luca-Walsh} performed extensive computational investigations to solve equation \eqref{eq:ab2}
    for nearly all pairs \((a, b)\) satisfying \(2 \leq b < a \leq 100\), leaving only 70 unresolved cases. 
After that, Le~\cite{Le} proved in 2009 that if $3 | b$, then the equation 
$(2^n - 1)(b^n - 1) = x^2$ has
no positive integer solution $(n, x)$.

Very recently, the authors \cite{He-Liu} completely solved the first example of the exponential Diophantine equation  
\begin{equation}\label{eq:ab}
(a^k - 1)(b^k - 1) = x^n
\end{equation}
in the case $(a, b) = (2,3)$. They proved that the equation
\begin{equation}\label{eq:23}
  (2^{k}-1)(3^{k}-1)=x^{n}
\end{equation}
has no solutions in positive integers $(x,k,n)$ with $k,n\geq 2$. A key step in our proof was to rewrite equation~\eqref{eq:23} in the form
\(
(X^q-1)(Y^q-1)=Z^q.
\)
Observe that this equation is structurally very similar to
\(
(x^k - 1)(y^k - 1) = z^k - 1,
\)
which was completely resolved in work of Bennett~\cite{Bennett:2007}, who
proved that for every integer \(k \ge 3\) the equation has no solutions in
integers \(x,y,z\) with  \(|z| \ge 2\). Indeed, the only formal difference between the equation
\(
  (X^q - 1)(Y^q - 1) = Z^q
\)
and Bennett's is the absence of the term $-1$ on the right hand side.
A careful inspection of Bennett's argument shows that his Diophantine
approximation method still applies with only minor modifications.

Hence, in this paper, we first analyze the equation
\begin{equation}\label{eq:qqq}
  (X^q-1)(Y^q-1)=Z^q,
\end{equation}
as a preliminary step towards solving equation~\eqref{eq:ab} in more generality. We prove:
\begin{theorem}\label{thm:qqq}
    The equation~\eqref{eq:qqq} has no solution in integers $X, Y, Z$ and odd prime $q$ with $1 < X \le Y$.
\end{theorem}

In view of existing results, in order to solve equation~\eqref{eq:ab} for fixed pairs $(a,b)$, it is enough to show that $q\mid k$. In this paper, we concentrate on the case $a=2$ and $b$ an odd integer in equation~\eqref{eq:ab}, and we establish the following result:

\begin{theorem}\label{thm:theorem1}
Let $b$ be an odd integer and $q$ an odd prime. 
If $q > 2\sqrt{2b}$, then the Diophantine equation
\begin{equation}\label{main equation}
\left(2^{k}-1\right)\left(b^{k}-1\right)=y^{q}
\end{equation}
admits no integer solutions $(k,y)$.
\end{theorem}

\begin{theorem}\label{thm:theorem2}
    When $b\leq10^6$ is an odd integer, the equation~\eqref{main equation} has no solution with positive integers $k, y \geq 2$ and  $q>\log_2(b+1)$ except for the following special cases:
    \begin{center}
\begin{tabular}{c|l}
$q$ & Remaining $b$ (odd, $3\le b<10^6$) \\
\hline
5  & $15,\ 17$ \\
11 & $1023,\ 1025$ \\
13 & $4095,\ 4097$ \\
17 & $t\cdot 2^{14}\pm 1\ \ (t=1,2,\dots,7)$ \\
\end{tabular}
\end{center}
\end{theorem}

\begin{theorem}\label{thm:2b} 
The Diophantine equation 
$$
(2^k-1)(b^k-1)=x^n, \quad k,n\ge 2 
$$
has no solution in positive integers $(k, x, n)$ for  $b = 5, 7,11, 13, 21,23,27,29$ except for $(2^2-1)(7^2-1)=12^2$. 
\end{theorem}

\section{Proof of Theorem~\ref{thm:qqq}}
Consider the Diophantine equation 
\begin{equation}\label{XYZ}
\left(X^q - 1\right)\left(Y^q - 1\right) = Z^q, \quad 1 < X \le Y, \quad q \text{ an odd prime}.     
\end{equation}

The case \(X=Y\) is impossible, since then \((X^q-1)^2=Z^q\) would imply \(X^q-1=U^q\) for some integer \(U\ge1\), hence \(X^q-U^q=1\). It remains to consider the case \(1<X<Y\). 
Our method is based on Bennett's work on rational approximation to algebraic
numbers~\cite{Bennett:1997} and on his proof that, for every integer
\(k \ge 3\), the Diophantine equation
\((x^k-1)(y^k-1)=z^k-1\) has no non-trivial integer solutions with \(|z| \ge 2\)~\cite{Bennett:2007}.

Let
\begin{equation}\label{abXYZ}
A:= X^q-1, \quad B:= Y^q-1.
\end{equation}
Then $ AB = Z^q.$
There exists a positive integer \( t \) such that
\[
XY = Z + t.
\]
Expanding the expression
\[
\left((AB)^{1 / q} + t\right)^q = (A+ 1)(B + 1),
\]
yields
\begin{equation}\label{eq:last}
q (AB)^{(q-1) / q} t + \binom{q}{2} (AB)^{(q-2) / q} t^2 + \cdots + t^q = A+ B + 1.
\end{equation}
We claim that
\[
  A < \binom{q}{2} (AB)^{(q-2) / q} t^2.
\]
If $q \ge 4$, this is immediate from $B \ge A$.  
If $q=3$ and $A \ge 3t^2(AB)^{1/3}$, then $B < A^{2}$, it follows that
\[
  3(AB)^{2/3}t > 3Bt > A+B +1,
\]
which is a contradiction to equation~\eqref{eq:last}. Therefore, 
\[
  q(AB)^{(q-1)/q}t < B,
\]
which yields
\begin{equation}\label{b lower bound}
  B>q^{q}A^{q-1}t^{q}. 
\end{equation}

Next, observe that from~\eqref{abXYZ}, we have
\[
  \left(\frac{XY}{Z}\right)^{q}- \frac{A+1}{A}
  = \frac{A+1}{AB}< \frac{2A}{Z^{q}}.
\]
It follows that
\begin{equation}\label{upper bound}
  \Bigl|\,\sqrt[q]{1+\frac1A}-\frac{XY}{Z}\Bigr|
  < \frac{2A}{q\cdot Z^{q}}
\end{equation}

Now we appeal to another result in order to deduce a lower bound which will contradict ~\eqref{upper bound}.
\begin{lemma}\cite{Bennett:1997}\label{lem:Be}
Let \(k,A\in\mathbb{Z}_{>0}\) with \(k\ge 3\). Define
\[
  \mu_n = \prod_{p\mid n} p^{1/(p-1)} ,
\] and suppose
\begin{equation}\label{k^k}
  \left(\sqrt{A}+\sqrt{A+1}\right)^{2(k-2)}
  >(k\mu_k)^{k}.
\end{equation}
Then
\[
  \left|\sqrt[k]{1+\frac{1}{A}}-\frac{p}{q}
  \right|>
  (8k\mu_k A)^{-1}\cdot q^{-\lambda},
\]
with
\begin{equation}\label{lamda}
  \lambda=1+\frac{
        \log\left(k\mu_k(\sqrt{A}+\sqrt{A+1})^{2}\right)
      }{
        \log\left(\tfrac1{k\mu_k}(\sqrt{A}+\sqrt{A+1})^{2}\right)
      }.
\end{equation}
and $\lambda<k$.
\end{lemma}

In our case, since \(A=X^{q}-1\), it is easy to verify that the inequality~\eqref{k^k} fails only 
when $(X,q)\in\{(2,3),(3,3)\}$. We assume that $(X,q)$ lies outside the set $\{(2,3),(3,3)\}$.  
Combining~\eqref{upper bound} with
Lemma~\ref{lem:Be} implies
\begin{equation}\label{size of Z}
  Z^{q-\lambda} < 16\mu_qA^{2},
\end{equation}
and thus,
\begin{equation}\label{b upper bound}
  B^{q-\lambda}<
    16^{q}\mu_q^{q}A^{q+\lambda}= 16^{q}q^{\frac{q}{q-1}}A^{q+\lambda}.
\end{equation}

From $A=X^{q}-1$, the
expression for $\lambda$ in~\eqref{lamda} shows that $\lambda$ is
monotonically decreasing in $X\ge 2$ (for $q\ge 7$), so
$\lambda<3.15$.  Therefore, ~\eqref{b upper bound} implies
\[
 B < 300\cdot A^{2.7},
\]
which contradicts~\eqref{b lower bound}.

\medskip
Similarly, if $q=5$ and $X\ge 3$, then $\lambda<2.8$, hence
\[
 B<1400\cdot A^{3.6},
\]
which again contradicts ~\eqref{b lower bound}.
However, there is no contradiction between the upper bound and lower bound for $B$ when $(X,q)=(2,q)$. We consider this case separately. Let $q=5$ and $X=2$, then
\begin{equation}
  31(Y^{5}-1)=Z^{5},
\end{equation}
and therefore
\[
  \left|\sqrt[5]{31}-\frac{Z}{Y}\right|<
    \frac{31^{1/5}}{5Y^{5}}.
\]
However, by Corollary~1.2 of \cite{Bennett:1997},
\[
  \left|\sqrt[5]{31}-\frac{Z}{Y}\right|    >    \frac{0.01}{Y^{2.83}},
\]
it follows that $Y\leq 6$. By a simple verification, the equation has no solution.

For the case $q=3$, we make a slight modification of the argument in
\cite{Bennett:2007}. Starting from the equation
\[
(X^{3}-1)Y^{3}-Z^{3}=X^{3}-1
\]
and writing $\alpha=(X^{3}-1)^{1/3}$, one obtains a very good rational
approximation $Z/Y$ to $\alpha$, with 
\begin{equation}\label{eq:basic-upper}
|\alpha-Z/Y|<\frac{A^{1/3}}{2.87Y^{3}}< \frac{1}{2Y^{2}} .   
\end{equation}
By Legendre’s
criterion, $Z/Y$ must be a convergent of the simple continued fraction of
$\alpha$.

Since $Z/Y<\alpha$, the index $j$ is even, which gives
\[
\Bigl|\alpha-\frac{p_j}{q_j}\Bigr|>\frac{1}{(a_{j+1}+2)q_j^2}.
\]
Combining this with \eqref{eq:basic-upper} and using $Y=q_j$, we obtain
\[
3A^{2/3}<Y<(a_{j+1}+2)A^{1/3}.
\]
Since $A=X^3-1$, this implies
\begin{equation}\label{eq:aj1-lb}
a_{j+1}>3X-3.
\end{equation}

For $X\ge 3$, the simple continued fraction of $\alpha=\sqrt[3]{X^3-1}$ begins
\[
a_0=X-1,\quad a_1=1,\quad a_2=3X^2-2,\quad a_3=1,\quad
a_4=X-2,\quad a_5=1.
\]
Thus, all odd partial quotients up to $a_5$ are equal to $1$. Since
\eqref{eq:aj1-lb} gives $a_{j+1}\ge 6$, our even index $j$ must satisfy
$j\ge 6$. Moreover, one computes the next partial quotients excluding $X \in \{2,3,5,7\}$,
\[
a_6=
\begin{cases}
(9X^2-4)/2,& X\equiv 0\pmod 2,\\[2pt]
(9X^2-3)/2,& X\equiv 1\pmod 2,
\end{cases}
\qquad
a_7=
\begin{cases}
1,& X\equiv 0\pmod 2,\\
2,& X\equiv 1\pmod 2.
\end{cases}
\]
Hence $a_7\le 2$, which is incompatible with \eqref{eq:aj1-lb} at $j=6$.
Therefore $j\ge 8$.

Using the explicit formulas for  $a_8$ and $q_8$ given in Bennett~\cite{Bennett:2007}, one finds that except for
\[
X\in\{2,3,5,7,9,11,15,17,19,21,25,27,31,37,41,47,57\},
\]
it follows that
\[
Y=q_j\ge q_8>5X^6.
\]
For the remaining values in this set with $X\ge 9$, one computes that
$a_9\le 10$, so again \eqref{eq:aj1-lb} forces $j\ge 10$, and hence
\[
Y\ge q_{10}>5X^6.
\]
For the four small remaining candidates $X\in\{2,3,5,7\}$, a direct
continued fraction computation shows that no convergent $p_j/q_j$ with
$q_j\le 5X^6$ yields an integer solution to
\[
(X^3-1)Y^3-Z^3=X^3-1.
\]
Therefore
\[
Y>5X^6
\]
holds for every $X\ge 2$. In particular,
\[
B=Y^3-1>125X^{18}-1>125A^6.
\]

Combining this with \eqref{b upper bound} for $q=3$ and for $X \geq 4$, we obtain
\[
A^{15-7\lambda}<2^{12}\cdot 3^{3/2}\cdot 5^{3\lambda-9},
\]
where
\[
\lambda=
1+\frac{\log\!\left(3\sqrt3\,(\sqrt A+\sqrt{A+1})^2\right)}
{\log\!\left(\frac{1}{3\sqrt3}(\sqrt A+\sqrt{A+1})^2\right)}.
\]
Since $A=X^3-1$, this is impossible for every $X\ge 16000$.

It remains to consider $4\le X\le 16000$. From equation~\eqref{size of Z},
we get
\[
Z < \left(8 \sqrt{3} A(A+1)\right)^{1 /(3-\lambda)} < 10^{31}.
\]
A direct computation of the continued fraction expansion of $\sqrt[3]{X^3-1}$
shows that no convergent $p_j/q_j$ with $p_j<10^{31}$ and $q_j>1$ satisfies
the necessary divisibility condition
\[
(X^3-1)q_j^3-p_j^3\mid (X^3-1).
\]
Hence, no solution exists for $4\le X\le 16000$.

Thus only $X\in\{2,3\}$ remain. For these, we use the explicit irrationality
measures in Corollary~1.2 of \cite{Bennett:1997}:
\[
\left|\sqrt[3]{7}-\frac{p}{q}\right|>\frac{0.08}{q^{2.70}},
\qquad
\left|\sqrt[3]{26}-\frac{p}{q}\right|>\frac{0.03}{q^{2.53}},
\]
for all positive integers $p$ and $q$. Combining these bounds with
\eqref{eq:basic-upper}, we obtain $Y\le 1172$ when $X=2$, and $Y\le 1859$
when $X=3$. A short computation shows that neither case yields a solution to
\eqref{XYZ}. Hence, the case $q=3$ cannot occur, and the proof of
Theorem~\ref{thm:qqq} is complete.

\section{Proof of Theorem~\ref{thm:theorem1}}

In this section, we obtain a conditional upper bound for $q$.
Let $p$ be a prime and $n\in\mathbb{Z}\setminus\{0\}$. The $p$-adic valuation $\nu_p(n)$
is the largest integer $e\ge 0$ such that $p^{e}\mid n$.  The following result, known as the Lifting-the-Exponent (LTE) Lemma \cite{LTE}, is a standard tool for estimating $p$-adic valuations
of exponential differences.

\begin{lemma}[Lifting-the-exponent Lemma \cite{LTE}]
\label{Lifting-the-exponent Lemma}
Let \( p \) be a prime, and let \( a \) and \( b \) be integers such that \( k \) is a positive integer. Suppose \( p \mid (a - b) \) and \( p \nmid ab \). Then, the \( p \)-adic valuation \( \nu_p \) of \( a^k - b^k \) is given by
\[
\nu_p(a^k - b^k) = 
\begin{cases}
\nu_p(a - b) + \nu_p(k), & \text{if } p \text{ is odd}, \\
\nu_2(a - b) , & \text{if } p = 2 \text{ and } k \text{ is odd}, \\
\nu_2(a^2 - b^2) + \nu_2\left(\frac{k}{2}\right), & \text{if } p = 2 \text{ and } k \text{ is even}.
\end{cases}
\]
\end{lemma}

\begin{lemma}
\label{lemma 2}
  Assume that $(k,b, y, q)$ is a positive integer solution of equation~\eqref{main equation} with $q>\log_2(b+1)$. We have

\begin{enumerate}
  \item $\nu_{2}(k) >q- \log_2(b+1) >0$,
  \item $\nu_{3}(k) >0$.
\end{enumerate}  
\end{lemma}

\begin{proof}

For point (1), since $b$ is odd, we have $2\mid (b^{k}-1)$, hence $2\mid y$. Thus, 
    \[
    b^k - 1 \equiv 0 \pmod{2^q}.
    \]
Applying Lemma~\ref{Lifting-the-exponent Lemma}, we obtain
    \[
    q \leq \nu_2\left(y^q\right) = \nu_2\left(b^k - 1\right) =
    \begin{cases}
        \nu_2(b-1),& \text{if } k \text{ is odd},\\
        \nu_2(b^2-1)+\nu_2(k/2),& \text{if } k \text{ is even.}
    \end{cases}
    \] 
  The bounds $q\leq  \nu_2(b-1)$ and 
   $q>\log_2(b+1)$ are in contradiction.    Therefore $ 2 \mid k$ and 
   \[ \nu_2(k) \ge q  +1 - \nu_2(b^2-1) \ge q-\log_2(b+1).
\]
Here we used $\min \{\nu_2(b+1), \nu_2(b-1)\} =1$. 
 
When $p=3$ in point (2), the condition \( 2^k - 1 \equiv 0 \pmod{3} \) implies that  $3 \mid y$. 
We thus have
    \[
   (2^k - 1)(b^k-1) \equiv 0 \pmod{3^q}.
    \]
If $3\mid b$, another application of Lemma~\ref{Lifting-the-exponent Lemma} to $4^{k/2}-1$ yields $3^{q - 1} \mid  k$.
Otherwise,  \[
\nu_3(b^k-1)=\nu_3(b^2-1)+\nu_3(k/2)=\nu_3(b-1)+\nu_3(b+1)+\nu_3(k)\leq \log_3(b+1)+\nu_3(k),
\] since the odd prime $p$ cannot divide both $b+1$ and $b-1$ simultaneously. 
Hence 
\[
\nu_3(k)\geq\frac{q-\log_3(b+1)-1}{2}.
\]
The right hand side is positive when $b\ge 7$. 
The remaining case that $b=5$ satisfies $\nu_3(b+1)=1$, can be checked directly. This completes the proof of Lemma~\ref{lemma 2}. 
 
\end{proof}

\begin{lemma}
\label{lemma 3}
Assume that \( (k, b,y, q) \) is a quadruple of positive integers satisfying
equation~\eqref{main equation}. Let $p$ be an odd prime. If \( (p - 1)\mid k \), then
$$
\nu_{p}(k)>\frac{1}{2}\left(q- p\cdot\frac{\log 2b}{2\log p} \right).
$$    
In particular, if we further assume that $p\le q$ and $q>2\sqrt{2b}$, then $p \mid k$.  
\end{lemma} 

\begin{proof}
  Let us begin by analyzing the \( p \)-adic valuation on both sides of equation~\eqref{main equation}. Assume that \( p - 1 \) divides \( k \). Then, by Fermat's Little Theorem, it follows that \( p \) divides both \( b^{p-1} - 1 \) and \( b^k - 1 \) for any integer \( b \) satisfying \( p \nmid b \). By applying Lemma~\ref{Lifting-the-exponent Lemma}, we obtain
\[
\nu_p\left(2^k - 1\right) = \nu_p\left(2^{p-1} - 1\right) + \nu_p\left(\frac{k}{p-1}\right) = \nu_p\left(2^\frac{p-1}{2} - 1\right) +\nu_p\left(2^\frac{p-1}{2} + 1\right)+ \nu_p(k),
\]

\[
\nu_p\left(b^k - 1\right) =\begin{cases}
    \nu_p\left(b^\frac{p-1}{2} - 1\right) + \nu_p\left(b^\frac{p-1}{2} + 1\right)+\nu_p(k), & \text{if } p \nmid b,\\
    0 , & \text{if } p \mid b.
\end{cases} 
\]
Next, we observe that
\[\nu_p\left(2^\frac{p-1}{2} - 1\right) +\nu_p\left(2^\frac{p-1}{2} + 1\right) < p \cdot \frac{\log 2}{2\log p}, \]
\[ \nu_p\left(b^\frac{p-1}{2} - 1\right) +\nu_p\left(b^\frac{p-1}{2} + 1\right) < p \cdot \frac{\log b}{2\log p}.
\]
As a consequence,
\[
\nu_p\left(\left(2^k - 1\right)\left(b^k - 1\right)\right) < p \cdot \frac{\log 2b }{2\log p} + 2 \nu_p(k).
\]
On the right hand side of equation~\eqref{main equation}, since \( p \mid y \), it follows that
\[
\nu_p\left(y^q\right) \geq q.
\]
Combining the inequalities derived above, we conclude
$$
\nu_{p}(k)>\frac{1}{2}\left(q- p\cdot\frac{\log 2b}{2\log p} \right).
$$ 
Assume that $p\le q$ and  $q>2\sqrt{2b}$. We have 
$$
\nu_{p}(k)>\frac{1}{2}\left(q-  p \cdot \frac{\log2b}{2\log p}\right) \geq \frac{1}{2}\left(q-  q \cdot \frac{\log2b}{2\log q}\right)>0,$$
which completes the proof of the lemma.
\end{proof}

\begin{lemma}\label{lemma 4}
Assume that \( (k, b,y, q) \) is a quadruple of positive integers solving
equation~\eqref{main equation}  with $q\ge 2\sqrt{2b}$, then we have \( q \mid k \).   
\end{lemma} 

\begin{proof}
    If \( q =3\), the result follows directly from Lemma~\ref{lemma 2} with $1+\log_2(b+1)<2\sqrt{2b}$. 
    Define \( P_n = \{ p_1=2,p_2=3, p_3, p_4, \ldots, p_n \} \) to be the set of all primes up to \( q \),  where \( p_3 = 5 < p_4 < \cdots < p_n = q \).
 To prove the result, it suffices to show that \( (p - 1) \mid k \) for each \( p \in P \), so that Lemma~\ref{lemma 3} can be applied.


Assume that for some \( i \), we already have \( (p_i-1) \mid k \). We aim to prove \mbox{\( (p_{i+1}-1) \mid k \)}. It suffices to establish that
\begin{equation}\label{2}
    \nu_r(p_{i+1} - 1) \leq \nu_r(k)
\end{equation}
for all \( r \in P_i \).

We first consider \( r = 2,3\). By  Lemma~\ref{lemma 2} and elementary monotonicity, 
\[
\nu_2(p_{i+1} - 1) \leq \frac{\log (p_{i+1} - 1)}{\log 2} \le \frac{\log (q-1)}{\log 2} < q-\log_2(b+1)\leq\nu_2(k) . 
\]

\[
\nu_3(p_{i+1} - 1) \leq \frac{\log (p_{i+1} - 1)}{\log 3} \le \frac{\log (q-1)}{\log 3} < \frac{q-\log_3(b+1)-1}{2}<\nu_3(k) . 
\]

Next, we consider \( r \geq 5 \). 
We distinguish two sub cases, according to whether \(r \le q^{1/2}\) or \(r > q^{1/2}\).

If \( r > q^{1/2} \), it follows that
   \[
   \nu_r(p_{i+1} - 1) \le \frac{\log(p_{i+1}-1)}{\log r} < \frac{\log q}{\log r} < \frac{\log r^2}{\log r} = 2.
   \]
   Since \(r \in  P_i = \{ 2,3, p_1, p_2, \ldots, p_i \} \), then \((r-1)\mid k\) by assumption. Thus, by Lemma~\ref{lemma 3}, we deduce that $r \mid k$.
   Therefore
   \[
   \nu_r(p_{i+1} - 1) \leq 1 \leq \nu_r(k).
   \]

If \( 5\leq r < q^{1/2} \), we claim that
   \[\frac{1}{2}\left(q- r\cdot\frac{\log 2b}{2\log r} \right)\geq \frac{\log (q-1)}{\log r}.
   \]
It suffices to prove    \[\frac{1}{2}\left(q- q^{\frac{1}{2}}\cdot\frac{\log \frac{q^2}{4}}{2\log q^{\frac{1}{2}}} \right)- \frac{\log (q-1)}{\log 5}\geq 0.
   \]
which indeed holds, as one verifies by examining the monotonicity of the left hand side. Therefore, 
we have shown that \( (p_{i+1} - 1) \mid k \). Consequently, by Lemma~\ref{lemma 3}, we conclude \( 
p_{i+1} \mid k \). This completes the induction, and hence every prime in the set \( P \) divides \( 
k \), including \( q \). This completes the proof.
\end{proof}
Now, assume that \( (k, b,y, q) \) is a quadruple of positive integers solving
equation~\eqref{main equation}  under the assumption of Theorem~\ref{thm:theorem1}.  By Lemma~\ref{lemma 4}, we have deduced \( q \mid k \). Define:
\[
(X, Y, Z) := \left(2^{k / q}, b^{k / q}, y\right);
\]
we then obtain a new triple of positive integers that satisfies the  Diophantine equation 
$$
(X^q-1)(Y^q-1)=Z^q.
$$
However, Theorem~\ref{thm:qqq} shows that this equation has no non-trivial positive integer solutions with odd prime $q$. 
This contradiction completes the proof of Theorem~\ref{thm:theorem1}. 

\section{Proof of Theorem~\ref{thm:theorem2}}
In this section, we sharpen the upper bound for $q$ for certain moderate values of the odd integer $b$. By Theorem~\ref{thm:theorem1}, we have shown that for every odd integer
\(3\le b<10^{6}\) there is no solution in which \(q\) is an odd prime with \(q>2828\).

\begin{lemma}
If $b$ is an odd integer with $b\le 10^6$, then the corresponding results in Lemmas~\ref{lemma 2}–\ref{lemma 4} hold for every prime $q\geq23$. Moreover, under these same conditions, the Diophantine equation \eqref{main equation} admits no solutions.
\end{lemma}

\begin{proof}
 Since $q\geq23>\log_2(b+1)$, we have $\nu_2(k)>1$ and $\nu_3(k)>0$. Therefore, Lemma~\ref{lemma 2} applies.
 
Then we investigate the exponents of prime divisors in $2^{k}-1$ and $b^{k}-1$. By Lemma~\ref{Lifting-the-exponent Lemma},
 \[
\nu_p\left(2^k - 1\right) = \nu_p\left(2^{p-1} - 1\right) + \nu_p\left(\frac{k}{p-1}\right) =\nu_p\left(2^{p-1} - 1\right) +  \nu_p(k),
\]
and $\nu_p\left(2^{p-1} - 1\right) =1 $ for  every prime $p<2828$ except for $p=1093$ where $\nu_{1093}\left(2^{1092} - 1\right)=2$. 
Similarly,
 \[
\nu_p\left(b^k - 1\right) = \nu_p\left(b^{p-1} - 1\right) + \nu_p\left(\frac{k}{p-1}\right) =\nu_p\left(b^{p-1} - 1\right) +  \nu_p(k).
\]
By computer calculation, we find that
$\nu_p\left(b^{p-1} - 1\right)\leq11$. Hence,
\begin{equation}\label{k lower bound}
    \nu_p(k)\geq \frac{q-\nu_p\left(b^{p-1} - 1\right)-\nu_p\left(2^{p-1} - 1\right)}{2}\geq 5,
\end{equation}
so Lemma~\ref{lemma 3} holds.

For Lemma~\ref{lemma 4}, assume that for some \( i \), we have \( (p_i-1) \mid k \). Our goal is to show that \( (p_{i+1}-1) \mid k \). It suffices to establish that
\begin{equation*}
    \nu_r(p_{i+1} - 1) \leq \nu_r(k)
\end{equation*}
for all \( r \in P_i \).
And it has been proved in~\eqref{k lower bound} that $\nu_r(k)\geq5$. 

We begin with the case \( r = 2,3\). By Lemma~\ref{lemma 2}, 
\[q-\log_2(b+1)\leq\nu_2(k). \] 
It is obvious that 
\[
\nu_2(p_{i+1} - 1) \leq \nu_2(q - 1) < q-\log_2(b+1) . 
\]
for all odd primes $q\ge 23$ and all $b\le 10^6$.
Similarly, 
\[
\nu_3(p_{i+1} - 1) \leq \nu_3(q - 1)  < \frac{q-\log_3(b+1)-1}{2}\leq\nu_3(k) . 
\]
for all odd primes $q\ge 23$ and $b\le 10^6$.

For \(r\ge5\), we have
\[
\nu_r(p_{i+1}-1)\leq \nu_5(q-1)\leq 4\leq \nu_r(k).
\]
Hence Lemma~\ref{lemma 4} holds, which means $q\mid k$. Using Theorem~\ref{thm:qqq}, we complete the proof of this lemma.
\end{proof}

Now, we prove Theorem~\ref{thm:theorem2}. 
\begin{proof}
There are a few possibilities for $3\leq q<23$.    By Lemma~\ref{lemma 2}, we have $\nu_2(k)>0$ and $\nu_3(k)>0$. Let \(p\) be a prime with \(3\le p\le q\).
We claim that if \((p-1)\mid k\) then \(p\mid k\). Since we already know that
\(3\mid k\), it remains to consider \(5\le p\le q<23\). Likewise,
\[
\nu_p\left(\left(2^k - 1\right)\left(b^k - 1\right)\right) =\nu_p\left((2^{p-1}-1)(b^{p-1}-1)\right)+2\nu_p(k)\geq q.
\]
Hence
\[
\nu_p(k)\geq \frac{q-\nu_p(2^{p-1}-1)-\nu_p(b^{p-1}-1)}{2}
\]
Since \(\nu_p\!\left(2^{p-1}-1\right)=1\), we only need to consider those \(b\) for which
\(\nu_p\!\left(b^{p-1}-1\right)\ge 4\) (as \(q\ge 5\)). A finite computer calculation shows that
there are no pairs \((b,p)\) with \(q-1-\nu_p\!\left(b^{p-1}-1\right)\le 0\) under this assumption.

Therefore, $q>\log_2(b+1)$, $3\mid k$ and $7\mid k$ when $q=3$ or $q=7$ since $2\mid k$ and $6\mid k$ respectively. Thus by
Theorem~\ref{thm:qqq} there is no solution to equation~\eqref{main equation}.

Now consider the case $q=5$. If $\nu_2(k)\geq 2$, then $4\mid k$ and hence $5\mid k$. Therefore, $\nu_2(k)=1$, which forces $\nu_2(b^2-1)=5n$, where $n\in \mathbb{Z}_{>0}$. In combination with $5=q>\log_2(b+1)$, the remaining unsolved cases are $b=15,17$.

Similarly, for $q=11$ and $q=13$ the remaining unsolved cases are $b=1023, 1025$ and $b=4095,4097$, respectively.

Next, consider the case $q=17$. If $\nu_2(k)\geq 4$, then $16\mid k$ and hence $17\mid k$. Therefore, $1\leq \nu_2(k)\leq 3$, which implies $15\leq \nu_2(b^2-1)\leq 17$ . In combination with $17=q>\log_2(b+1)$, the remaining unsolved cases are $b=t\cdot2^{14}\pm 1$ where $t=1,2,...,7$.

Finally, consider $q=19$. We have $\nu_3(k)\geq\frac{q-\log_3(b+1)-1}{2}=\frac{19-\log_3(b+1)-1}{2}>2$. Then $18\mid k$ and hence $19\mid k$.
By Theorem~\ref{thm:qqq}, there is no solution for equation~\eqref{main equation}.

The remaining unresolved instances, with  an odd prime $q$ and odd positive \mbox{\(b<10^{6}\)} where $q>\log_2(b+1)$ are summarized in the table. This completes the proof of Theorem~\ref{thm:theorem2}. 

\end{proof}

\section{Proof of Theorem~\ref{thm:2b}}
Assume that $(k,x,n)$ is a solution in positive integers of
\begin{equation}\label{eq:2bn}
(2^k-1)(b^k-1)=x^n,\quad k,n>1. 
\end{equation}
First, consider the case $n=2$. By Theorem~3.1 in~\cite{Luca-Walsh}, any solution of
\[
(2^k-1)(b^k-1)=x^2
\]
with odd \(3\le b<30\) must satisfy \(k=2\). Hence
\(
3(b^2-1)=x^2.
\)
A direct check shows that this occurs only for \(b=7\). 
Thus, we may assume that $n>2$ and $2\nmid n$.    Let $q$ be the least prime
divisor of $n$ and put $y=x^{n/q}$. Then $(k,y,q)$ is a solution in positive
integers of equation ~\eqref{main equation}, namely
$$
(2^k-1)(b^k-1) = y^q. 
$$
We now restrict to the values $b=5,7,11,13,21,23,27,29$. By Theorem~\ref{thm:theorem2} we have
\[
  q\le \log_2(b+1)<4,
\]
hence $q$ must be $3$, since $q$ is an odd prime.

For $b=5$, Theorem~\ref{thm:theorem2} yields
\[
  q<\log_2 6<3,
\]
which is impossible for an odd prime $q$. 

In the cases $b=7,11,13,21,23,27,29$, comparing the $2$-adic valuations on both sides of
\[
  (2^k-1)(b^k-1)=y^3
\]
shows that we must have $2\mid k$. 
Applying Lemma~\ref{Lifting-the-exponent Lemma} with $p=3$ gives
\[
  \nu_3(2^k-1)
    = \nu_3(4-1)+\nu_3(k/2)
    = 1+\nu_3(k)
\]
and
\[
  \nu_3(b^k-1)
    = \begin{cases} \nu_3(b-1)+\nu_3(k) & \text { if } b=7,13; \\
    \nu_3(b^2-1)+\nu_3(k) & \text { if } b=11,23,29;\\
    0& \text { if } b=21,27.
    \end{cases}
\]
Consequently,
\[
  \nu_3\bigl((2^k-1)(b^k-1)\bigr)
 = \begin{cases} 2+2\nu_3(k) & \text { if } b=7,13; \\
    2+2\nu_3(k) & \text { if } b=11,23,29;\\
    1+\nu_3(k)& \text { if } b=21,27.
    \end{cases}
\]
On the other hand,
\[
  \nu_3(y^q)=\nu_3(y^3)=3\,\nu_3(y),
\]
and the equality of both sides implies
\[
  3 = q \le \nu_3(y^q) = \begin{cases} 2+2\nu_3(k) & \text { if } b=7,13; \\
    2+2\nu_3(k) & \text { if } b=11,23,29;\\
    1+\nu_3(k)& \text { if } b=21,27.
    \end{cases},
\]
so $3\mid k$.
This contradicts Theorem~\ref{thm:qqq}, which states that the equation
$(X^q-1)(Y^q-1)=Z^q$ has no integer solution with $1<X\le Y$ and $q$ an odd
prime. 

Combining all the above cases, we conclude that for $b=5,7,11,13,21,23,27,29$ the Diophantine
equation
\[
  (2^k-1)(b^k-1)=x^n,\qquad n\ge 2,
\]
admits no solution in positive integers $(k,x,n)$ except for $(2^2-1)(7^2-1)=12^2$. This completes the proof of
Theorem~\ref{thm:2b}.

\section*{Acknowledgments} The authors express their gratitude to Prof. Preda Mih\u{a}ilescu for carefully reading a preliminary version of this paper and for his valuable comments and suggestions.
 The first   author acknowledges the support of the China Scholarship Council program (Project ID: 202106310023).  The second author was supported by the National Natural Science Foundation of China (Grant No. 12161001). 
 
\end{document}